\documentclass[12pt]{article}

\usepackage{amssymb,amsmath,amsthm}
\usepackage{epsfig}
\usepackage{breqn}
\usepackage{rotating}
\usepackage{amsfonts}
\usepackage{setspace}
\usepackage{fullpage}
\usepackage{enumitem}
\usepackage{bbold} 
\usepackage{comment}
\usepackage{pgf,tikz}
\usepackage{mathtools}  
\usepackage{hyperref}
\mathtoolsset{showonlyrefs} 
\usepackage{authblk}

\newtheorem{thm}{Theorem}
\newtheorem{claim}{Claim}
\newtheorem{lem}{Lemma}

\newtheorem{defi}{Definition}
\newtheorem{notation}{Notation}

\def\h{\mathcal H}
\def\P{\mathcal P}
\def\Q{\mathcal Q}

\newcommand{\abs}[1]{\left\lvert{#1}\right\rvert}

\begin{document}
\title{Avoiding long Berge cycles, the missing cases \\ $k=r+1$ and $k = r+2$}

\author[1]{Beka Ergemlidze}
\author[1,2]{Ervin Gy\H{o}ri} 
\author[1]{Abhishek Methuku}
\author[1]{Nika Salia}
\author[2]{Casey Tompkins}
\author[1,3]{Oscar Zamora} 
\affil[1]{Central European University, Budapest.\par
 \texttt{ \{abhishekmethuku,beka.ergemlidze\}@gmail.com, \newline Salia\char`_Nika@phd.ceu.edu, oscarz93@yahoo.es}     }
\affil[2]{Alfr\'ed R\'enyi Institute of Mathematics, Hungarian Academy of Sciences. \newline
 \texttt{gyori.ervin@renyi.mta.hu, ctompkins496@gmail.com}}
 \affil[3]{Universidad de Costa Rica, San Jos\'e.}

\maketitle

\begin{abstract}
The maximum size of an $r$-uniform hypergraph without a Berge cycle of length at least $k$ has been determined for all $k \ge r+3$ by F\"uredi, Kostochka and Luo and for $k<r$ (and $k=r$, asymptotically) by Kostochka and Luo.  In this paper, we settle the remaining cases: $k=r+1$ and $k=r+2$, proving a conjecture of F\"uredi, Kostochka and Luo.
\end{abstract}


Given a hypergraph $\h$, let $V(\h)$ and $E(\h)$ denote the set of vertices and hyperedges of $\h$, respectively, and let $e(\h) = \abs{E(\h)}$.  A hypergraph is called $r$-uniform if all of its hyperedges have size $r$.  For convenience, we refer to an $r$-uniform hypergraph as an $r$-graph. Berge introduced the following definitions of a cycle and a path in a hypergraph.   


\begin{defi}
 A \emph{Berge cycle} of length $l$ in a hypergraph is a set of $l$ distinct vertices $\{v_1, \ldots, v_l\}$ and $l$ distinct hyperedges $\{e_1, \ldots, e_l\}$ such that $\{v_i, v_{i+1}\} \subseteq e_i$ with indices taken modulo $l$. 

A \emph{Berge path} of length $l$ in a hypergraph is a set of $l + 1$ distinct vertices ${v_1, \dots, v_{l+1}}$ and $l$ distinct hyperedges ${e_1, \dots, e_l}$ such that $\{v_i,v_{i+1}\} \subseteq e_i$ for all $1 \le i \le l$. We say that such a Berge path is between $v_1$ and $v_{l+1}$. 
\end{defi}

\begin{notation}
Let $\mathcal H$ be a hypergraph. Then its \emph{2-shadow}, $\partial_2 \h$, is the collection of pairs that lie in some hyperedge of $\h$. Given a set $S \subseteq V(\h)$, the subhypergraph of $\h$ induced by $S$ is denoted by $\h[S]$. 

We say $\h$ is \emph{connected} if $\partial_2(\h)$ is a connected graph. A hyperedge $h \in E(\h)$ is called a \emph{cut-hyperedge} of $\h$ if $\h \setminus \{h\} := (V(\h), E(\h) \setminus \{h\})$ is not connected. 

When we say $D$ is a block of $\partial_2(\h)$, we may either mean $D$ is the vertex-set of the block, or $D$ is the edge-set of the block depending on the context.
\end{notation}

\section{Background and our results}

Gy\H{o}ri, Katona and Lemons extended the well-known Erd\H{o}s-Gallai theorem to hypergraphs by showing the following.
\begin{thm}[Gy\H{o}ri, Katona, Lemons \cite{GyoKaLe}]
\label{GKL}
	Let $\h$ be an $r$-uniform hypergraph with no Berge path of length $k$.  If $k>r+1>3$, we have
	\begin{displaymath}
	e(\h) \le \frac{n}{k} \binom{k}{r}.
	\end{displaymath}
	If $r \ge k>2$, we have
	\begin{displaymath}
	e(\h) \le \frac{n(k-1)}{r+1}.
	\end{displaymath}
\end{thm}
For the case $k=r+1$, Gy\H{o}ri, Katona and Lemons conjectured that the upper bound should have the same form as the $k>r+1$ case. This was settled by Davoodi, Gy\H{o}ri, Methuku and Tompkins \cite{DavoodiGMT} who showed the following. 

\begin{thm}[Davoodi, Gy\H{o}ri, Methuku, Tompkins \cite{DavoodiGMT}]
\label{DavoodiGMTthm}
Fix $k = r+1>2$ and let $\h$ be an $r$-uniform hypergraph containing no Berge path of length $k$.  Then,
	\begin{displaymath}
	e(\h) \le \frac{n}{k} \binom{k}{r} = n.
	\end{displaymath}
\end{thm}

The bounds in the above two theorems are sharp for each $k$ and $r$ for infinitely many $n$. Gy\H{o}ri, Methuku, Salia, Tompkins and Vizer \cite{GyoriMSTV} proved a significantly smaller upper bound on the maximum number of hyperedges in an $n$-vertex $r$-graph with no Berge path of length $k$ under the assumption that it is connected. Their bound is asymptotically exact when $r$ is fixed and $k$ and $n$ are sufficiently large. The notion of Berge cycles and Berge paths was generalized to arbitrary Berge graphs $F$ by Gerbner and Palmer in \cite{gp1}, and the ($3$-uniform) Tur\'an number of Berge-$K_{2,t}$ was determined asymptotically in \cite{GMM}. The general behaviour of the Tur\'an number of Berge-$F$, as the uniformity increases, was studied in \cite{GMTthreshold}. 

\vspace{2mm}

Recently, F\"uredi, Kostochka and Luo \cite{furedi2018avoiding} proved exact bounds similar to Theorem~\ref{GKL} for hypergraphs avoiding long Berge cycles.

\begin{thm}[F\"uredi, Kostochka, Luo \cite{furedi2018avoiding}]
\label{FKLuo}
Let $r \ge 3$ and $k \ge r+3$, and suppose $\h$ is an $n$-vertex $r$-graph with no Berge cycle of length $k$ or longer. Then $e(\h) \le \frac{n-1}{k-2}\binom{k-1}{r}$. Moreover, equality is achieved if and only if $\partial_2(\h)$ is connected and for every block $D$ of $\partial_2(\h)$, $D = K_{k-1}$ and $\h[D] = K_{k-1}^r$. 
\end{thm}

Moreover, Kostochka and Luo \cite{KostochkaLuo} found exact bounds for $k \le r-1$ and asymptotic bounds for $k=r$. Let us remark that their asymptotic bound in the case $k=r$ also follows from Theorem \ref{r+1} stated below. (More recently, extending \cite{furedi2018avoiding}, F\"uredi, Kostochka, Luo \cite{furedi2018avoiding2} proved exact bounds and determined the extremal examples for all $n$ when $k \ge r+4$.)

\vspace{2mm}

The two cases $k = r+2$ and $k = r+1$ remained open. For the case $k = r+2$, F\"uredi, Kostochka and Luo conjectured \cite{furedi2018avoiding} that a similar statement as that of Theorem \ref{FKLuo} holds and mentioned the answer is unknown for the case $k = r+1$. In this paper, we prove this conjecture.

\begin{thm}
\label{r+2}
Let $r \ge 3$ and $n \ge 1$, and suppose $\h$ is an $n$-vertex $r$-graph with no Berge cycle of length $r+2$ or longer. Then $e(\h) \le \frac{r+1}{r}(n-1)$. Moreover, equality is achieved if and only if $\partial_2(\h)$ is connected and for every block $D$ of $\partial_2(\h)$, $D = K_{r+1}$ and $\h[D] = K_{r+1}^r$. 
\end{thm}

In the case $k = r+1$, we prove the following exact result, and characterize the extremal examples.

\begin{thm}
\label{r+1}
Let $r \ge 3$ and $n \ge 1$, and suppose $\h$ is an $n$-vertex $r$-graph with no Berge cycle of length $r+1$ or longer. Then $e(\h) \le n-1$. Moreover, equality is achieved if and only if $\partial_2(\h)$ is connected and for every block $D$ of $\partial_2(\h)$, $D = K_{r+1}$ and $\h[D]$ consists of $r$ hyperedges. 
\end{thm}

Note that Theorem \ref{r+1} easily implies Theorem \ref{DavoodiGMTthm}. In fact, it gives the following stronger form. Here we quickly prove this implication.

\begin{thm}
Fix $k = r+1>2$ and let $\h$ be an $r$-uniform hypergraph containing no Berge path of length $k$.  Then, $e(\h) \le \frac{n}{k} \binom{k}{r} = n.$ Moreover, equality holds if and only if each connected component $D$ of $\partial_2(\h)$ is $K_{r+1}$, and $\h[D] = K_{r+1}^r$.
\end{thm}

\begin{proof}
We proceed by induction on $n$. The base cases $n \le r+1$ are easy to check.
Let $\h$ be an $r$-uniform hypergraph containing no Berge path of length $k = r+1$ such that $e(\h) \ge n$. Then by Theorem \ref{r+1}, $\h$ contains a Berge cycle $\mathcal C$ of length $r+1$ or longer. $\mathcal C$ must be of length exactly $r+1$, otherwise it would contain a Berge path of length $r+1$. Let $v_1, \ldots, v_{r+1}$ and $e_1, \ldots, e_{r+1}$ be the vertices and edges of $\mathcal C$ where $\{v_i, v_{i+1}\} \subseteq e_i$ (indices are taken modulo $r+1$). For any $i$ with $1 \le i \le r+1$, if $e_i$ contains a vertex $v \not \in \{v_1, \ldots, v_{r+1}\}$, then $v_{i+1} e_{i+1} v_{i+2} e_{i+2} \ldots e_{i-1} v_i e_i v$ forms a Berge path of length $r+1$ in $\h$, a contradiction. Therefore, all of the edges $e_i$ (for $1 \le i \le r+1$) are contained in the set $S := \{v_1, \ldots, v_{r+1}\}$. That is, $\h[S] = K_{r+1}^r$. It is easy to see that $S$ forms a connected component in $\partial_2(\h)$ because if any hyperedge $h$ of $\h$ (with $h \not \in \mathcal C$) contains a vertex of $\mathcal C$, then  $\mathcal C$ can be extended to form a Berge path of length $r+1$. 

Let $S_1, S_2, \ldots, S_t$ be the vertex sets of connected components of $\partial_2(\h)$. As noted before, one of them, say $S_1$, is equal to $S$. We delete the vertices of $S_1$ from $\h$ to form a new hypergraph $\h'$; note that $\abs{V(\h')} = \abs{V(\h)} -(r+1)$ and $\abs{E(\h')} = \abs{E(\h)} - (r+1)$ and the connected components of $\partial_2(\h')$ are $S_2, \ldots, S_t$.
By induction $\abs{E(\h')} \le \abs{V(\h')}$. Thus $\abs{E(\h)} = \abs{E(\h')} + (r+1) \le \abs{V(\h')} + (r+1) =  \abs{V(\h)}$. Moreover, if $\abs{E(\h)} = \abs{V(\h)}$, then $\abs{E(\h')} = \abs{V(\h')}$, so by the induction hypothesis each connected component $S_i$ ($i \ge 2$) of $\partial_2(\h')$ is $K_{r+1}$, and $\h'[S_i] = K_{r+1}^r$, proving the theorem.
\end{proof}

\vspace{2mm}

\textbf{Structure of the paper:} In Section \ref{basic_lemmas}, we prove some basic lemmas which are used in our proofs. In Section \ref{sec:proofforr+2case}, we prove Theorem \ref{r+2},  and in Section \ref{sec:proofforr+1case}, we prove Theorem \ref{r+1}.

\section{Basic Lemmas}
\label{basic_lemmas}

We will use the following two lemmas.

\begin{lem}
\label{r_path_between_any_pair}
For any $r \ge 3$, if a set $S$ of size $r+1$ contains $r$ hyperedges of size $r$, then between any two vertices $u, v \in S$, there is a Berge path of length $r$ consisting of these hyperedges. 
\end{lem}
\begin{proof}
Let $\h$ be the hypergraph consisting of $r$ hyperedges on $r+1$ vertices. First notice that for any pair of vertices $x,y \in S$, the number of hyperedges $h \subset S$ such that $\{x,y\} \not \subset h$ is at most 2. (Indeed, there is at most one hyperedge that does not contain $x$ and at most one hyperedge that does not contain $y$.) This means that every pair $x,y \in S$ is contained in some hyperedge, as there are at least 3 hyperedges contained in $S$. In other words, $\partial_2(\h) =K_{r+1}$.

Consider an arbitrary path $x_1x_2,\ldots,x_{r+1}$ of length $r$ in the $\partial_2(\h)$ connecting $u=x_1$ and $v=x_{r+1}$. We want to show that there are distinct hyperedges containing the pairs $x_ix_{i+1}$ for each $1 \le i \le r$. To this end, we consider an auxiliary bipartite graph with pairs $\{x_1x_2, x_2x_3, \ldots, x_rx_{r+1}\}$ in one class and the $r$ hyperedges $h \subset S$ in the other class, and a pair is connected to a hyperedge if it is contained in the hyperedge. We will show that Hall's condition holds: As noted before, every pair is contained in a hyperedge. Given any two distinct pairs $x_ix_{i+1}$ and $x_jx_{j+1}$, there is at most one hyperedge that does not contain either of them; i.e., at least $r-1$ hyperedges contain one of them. Thus we need $2 \le r-1$ for Hall's condition to hold, but this is true as we assumed $r \ge 3$. Moreover, if we take any $3 \le j \le r$ distinct pairs, then every hyperedge contains one of them. Therefore, we need $j \le r$, but this is true by assumption. This finishes the proof of the lemma.
\end{proof}

\begin{lem}
\label{r-1pathbetweenanypair}
For any $r \ge 4$, if a set $S$ of size $r+1$ contains $r-1$ hyperedges of size $r$, then between any two vertices $u, v \in S$, there is a Berge path of length $r-1$ consisting of these hyperedges. 
\end{lem}

\begin{proof}
The proof is similar to that of Lemma \ref {r_path_between_any_pair}. Let $\h$ be the hypergraph consisting of $r-1$ hyperedges on $r+1$ vertices. First notice that for any pair of vertices $x,y \in S$, the number of hyperedges $h \subset S$ such that $\{x,y\} \not \subset h$ is at most 2. This means that every pair $x,y \in S$ is contained in some hyperedge, as there are at least $r-1 \ge 3$ hyperedges contained in $S$. In other words, $\partial_2(\h) = K_{r+1}$.

Consider an arbitrary path $x_1x_2\ldots x_{r}$ of length $r-1$ in the $\partial_2(\h)$ connecting $u=x_1$ and $v=x_r$. We want to show that there are distinct hyperedges containing the pairs $x_ix_{i+1}$ for each $1 \le i \le r-1$. To this end, we consider an auxiliary bipartite graph with pairs $\{x_1x_2, x_2x_3, \ldots, x_{r-1}x_r\}$ in one class and the $r-1$ hyperedges $h \subset S$ in the other class, and a pair is connected to a hyperedge if it is contained in the hyperedge. We show that Hall's condition holds: As noted before, every pair is contained in a hyperedge. Given any two distinct pairs $x_ix_{i+1}$ and $x_jx_{j+1}$, there is at most one hyperedge that does not contain either of them; i.e., at least $r-2$ hyperedges contain one of them. Thus we need $2 \le r-2$ for Hall's condition to hold, but this is true as we assumed $r \ge 4$. Moreover, if we take any $3 \le j \le r-1$ distinct pairs, then every hyperedge contains one of them. Therefore, we need $j \le r-1$ for Hall's condition to hold, and this is true by assumption. This finishes the proof of the lemma.
\end{proof}

\section{Proof of Theorem \ref{r+2} ($k = r+2$)}
\label{sec:proofforr+2case}

We will prove the theorem by induction on $n$. For the base cases, note that if $1 \le n \le r$ then the statement of the theorem is trivially true. If $n = r+1$, the statement is true since there are at most $r+1$ hyperedges of size $r$ on $r+1$ vertices. Moreover, equality holds if and only if $\h = K_{r+1}^r$. 

We will show the statement is true for $n \ge r+2$ assuming it is true for all smaller values. Let $\h$ be an $r$-uniform hypergraph on $n$ vertices having no Berge cycle of length $r+2$ or longer. We show that we may assume the following two properties hold for $\h$.

\begin{enumerate}
\item[(1)] For any set $S \subseteq V(\h)$ of vertices, the number of hyperedges of $\h$ incident to the vertices of $S$ is at least $|S|$. 

Indeed, suppose there is a set $S \subseteq V(\h)$ with fewer than $|S|$ hyperedges incident to the vertices of $S$. If $\abs{S}=n$ we immediately have the required bound on $e(\h)$, so assume $n>\abs{S}$.  We can delete the vertices of $S$ from $\h$ to obtain a new hypergraph $\h'$ on $n-|S|$ vertices. By induction, $\h'$ contains at most $\frac{r+1}{r}(n-|S|-1)$ hyperedges, so $\h$ contains less than $\frac{r+1}{r}(n-1-|S|)+|S| < \frac{r+1}{r}(n-1)$ hyperedges, as desired. 

\item[(2)] There is no cut-hyperedge in $\h$.

Indeed, if $h \in E(\h)$ is a cut-hyperedge, then $\partial_2(\h \setminus \{h\})$ is not a connected graph, so there are non-empty disjoint sets $V_1$ and $V_2$ such that $V(\h) = V_1 \cup V_2$, and there are no edges of $\partial_2(\h \setminus \{h\})$ between $V_1$ and $V_2$. So both hypergraphs $\h[V_1]$ and $\h[V_2]$ do not contain a Berge cycle of length $r+2$ or longer. By induction, $e(\h[V_1]) \le \frac{r+1}{r}(|V_1|-1)$ and $e(\h[V_2]) \le \frac{r+1}{r}(|V_2|-1)$. In total, $e(\h) = e(\h[V_1]) + e(\h[V_2]) + 1 \le \frac{r+1}{r}(|V_1|+|V_2|-2)+1 < \frac{r+1}{r}(|V(\h)|-1)$, as desired. 



\end{enumerate}

Consider an auxiliary bipartite graph $B$ consisting of vertices of $\h$ in one class and hyperedges of $\h$ on the other class. Then property (1) shows that Hall's condition holds. Therefore, there is a perfect matching in $B$. In other words, there exists an injection ${f:V(\h)\to E(\h)}$ such that $v \in f(v)$. 

Given an injection $f:V(\h)\to E(\h)$ with $v \in f(v)$, let $\P_f$ be a longest Berge path of the form $v_1f(v_1)v_2f(v_2)\ldots v_{l-1}f(v_{l-1})v_l$  where for each $1 \le i \le l-1$, $v_{i+1} \in f(v_i)$. Moreover, among all injections $f:V(\h) \to E(\h)$ with $v \in f(v)$, suppose $\phi:V(\h) \to E(\h)$ is an injection for which the path $\P_{\phi} = v_1\phi(v_1)v_2\phi(v_2) \ldots v_{l-1}\phi(v_{l-1})v_l$ is a longest path. 

\begin{claim}
\label{r+1_set}
$\phi(v_l) \subset \{v_{l-r}, v_{l-r+1}, \ldots, v_{l-1}, v_l\}$.
\end{claim}
\begin{proof}
First notice that if $\phi(v_l)$ contains a vertex $v_i \in \{v_1,v_2, \ldots, v_{l-r-1}\}$, then the Berge cycle $v_i\phi(v_i)v_{i+1}\phi(v_{i+1}) \ldots v_l\phi(v_l)v_i$ is of length $r+2$ or longer, a contradiction. Moreover, if $\phi(v_l)$ contains a vertex $v \not \in \{v_1,v_2, \ldots, v_{l}\}$, then $\P_{\phi}$ can be extended to a longer path $v_1\phi(v_1)v_2\phi(v_2) \ldots v_{l-1}\phi(v_{l-1})v_l\phi(v_l)v$, a contradiction. This completes the proof of the claim.
\end{proof}

By Claim \ref{r+1_set}, we know that $\phi(v_l) = \{v_{l-r}, v_{l-r+1}, \ldots, v_{l-1}, v_l\} \setminus \{v_j\}$ for some $l-r \le j \le l-1$. 

\begin{claim}
\label{finding_many_r-sets}
For any $i \in \{l-r, l-r+1, \ldots, l\} \setminus \{j\}$, we have $\phi(v_i) \subset \{v_{l-r}, v_{l-r+1}, \ldots, v_{l-1}, v_l\}$.
\end{claim} 
\begin{proof}
When $i =l$, we know the statement is true. Suppose $i \in \{l-r, l-r+1, \ldots, l-1\} \setminus \{j\}$. Let us define a new injection $\psi: V(\h) \to E(\h)$ as follows: $\psi(v) = \phi(v)$ for every $v \not \in \{v_1,v_2, \ldots, v_l\}$, and for every $v \in \{v_1, v_2, \ldots, v_{i-1}\}$. Moreover, let $\psi(v_i) = \phi(v_l)$ and $\psi(v_k) = \phi(v_{k-1})$ for each $l \ge k \ge i+1$. 

Now consider the Berge path $v_1\phi(v_1)v_2\phi(v_2) \ldots v_i\phi(v_l)v_l\phi(v_{l-1}) \ldots v_{i+2}\phi(v_{i+1})v_{i+1}$, equivalently 
$v_1\psi(v_1)v_2\psi(v_2) \ldots v_i\psi(v_i)v_l\psi(v_{l}) \ldots v_{i+2}\psi(v_{i+2})v_{i+1}$. This path has the same length as $\P_{\phi}$, so it is also a longest path. Moreover, notice that the sets of last $r+1$ vertices of both paths are the same. Thus we can apply Claim \ref{r+1_set} to conclude that $\phi(v_i) = \psi(v_{i+1}) \subset \{v_{l-r}, v_{l-r+1}, \ldots, v_{l-1}, v_l\}$, as desired. 
\end{proof}

Claim \ref{finding_many_r-sets} shows that there are $r$ hyperedges (each of size $r$) contained in the set $S := \{v_{l-r}, v_{l-r+1},\ldots,$ $v_{l-1}, v_l\}$ of size $r+1$. We will apply Lemma \ref{r_path_between_any_pair} to $S$.

\begin{claim}
\label{finding_one_block}
The set $S= \{v_{l-r}, v_{l-r+1}, \ldots, v_{l-1}, v_l\}$ induces a block of $\partial_2(\h)$.
\end{claim}
\begin{proof}
Since the set $S = \{v_{l-r}, v_{l-r+1}, \ldots, v_{l-1}, v_l\}$ contains $r \ge 3$ hyperedges every pair $x, y \in S$ is contained in some hyperedge. Thus $\partial_2(\h[S]) = K_{r+1}$. Consider a (maximal) block $D$ of $\partial_2(\h)$ containing $S$. 

Suppose $D$ contains a vertex $t \not \in S$. Then since $D$ is 2-connected, there are two paths $P_1, P_2$ in $\partial_2(\h)$ between $t$ and $S$, which are vertex-disjoint besides $t$. Let $V(P_1) \cap S = \{u\}$ and $V(P_2) \cap S = \{v\}$. For each edge $xy \in E(P_1) \cup E(P_2)$, fix an arbitrary hyperedge $h_{xy}$ of $\h$ containing $xy$. It is easy to see that a subset of the hyperedges $\{h_{xy} \mid xy \in E(P_1) \cup E(P_2)\}$ forms a Berge path $\P$ between $u$ and $v$.

On the other hand, by Lemma \ref{r_path_between_any_pair}, there is a Berge path $\P'$ of length $r$ between $u$ and $v$ consisting of the $r$ hyperedges contained in $S$. Note that $\P$ and $\P'$ do not share any hyperedges (indeed, each hyperedge of $\P$ contains a vertex not in $S$, while hyperedges of $\P'$ are contained in $S$). Therefore, $\P \cup \P'$ forms a Berge cycle of length $r+2$ or longer unless $\P$ consists of only one hyperedge, say $h$. Note that $h$ contains a vertex $x \not \in S$ and $u, v \in h$; moreover by property (2), $h$ is not a cut-hyperedge of $\h$. So after deleting $h$ from $\h$, the hypergraph $\h \setminus \{h\}$ is still connected -- so there is a (shortest) Berge path $\Q$ in $\h \setminus \{h\}$ between $x$ and a vertex $s \in S$ (note that the hyperedges of $\Q$ are not contained in $S$). The vertex $s$ is different from either $u$ or $v$, say $s \not = u$ without loss of generality. By Lemma \ref{r_path_between_any_pair}, there is a Berge path $\Q'$ of length $r$ between $s$ and $u$ (consisting of hyperedges contained in $S$). Then, $\Q, \Q'$ and $h$ form a Berge cycle of length at least $r+2$, a contradiction. Therefore, $D$ contains no vertex outside $S$; thus $S$ induces a block of $\partial_2(\h)$, as required.
\end{proof}

Let $D_1, D_2, \ldots, D_p$ be the unique decomposition of $\partial_2(\h)$ into $2$-connected blocks. Claim \ref{finding_one_block} shows that one of these blocks, say $D_1$, is induced by $S$. Let us contract the vertices of $S$ to a single vertex, to produce a new hypergraph $\h'$. Then it is clear that the block decomposition of $\partial_2(\h')$ consists of the blocks $D_2, \ldots, D_p$. So $\h'$ does not contain any Berge cycle of length $r+2$ or longer, as well; moreover $|V(\h')| = |V(\h)|-r$. Thus, by induction, we have  $e(\h') \le \frac{r+1}{r}(|V(\h')|-1)$. Therefore, $$e(\h) \le \frac{r+1}{r}(|V(\h')|-1)+(r+1) = \frac{r+1}{r}(|V(\h)|-r-1)+(r+1) = \frac{r+1}{r}(|V(\h)|-1).$$

Now if $e(\h) =  \frac{r+1}{r}(|V(\h)|-1)$, then we must have $e(\h') = \frac{r+1}{r}(|V(\h')|-1)$ and $S$ must contain all $r+1$ subsets of size $r$ (i.e., $\h[S] = \h[D_1] = K_{r+1}^r$). Moreover, since equality holds for $\h'$, by induction, $\partial_2(\h')$ is connected and for each block $D_i$ (with $2 \le i \le p$) of $\partial_2(\h')$, $D_i = K_{r+1}$ and $\h'[D_i]= K_{r+1}^r$. This means that for every block $D$ of $\partial_2(\h)$, we have $D = K_{r+1}$ and $\h[D] = K_{r+1}^r$, completing the proof.

\section{Proof of Theorem \ref{r+1} ($k = r+1$)}
\label{sec:proofforr+1case}

The proof is similar to that of Theorem \ref{r+2} but there are many important differences. 

We use induction on $n$. For the base cases, notice that the statement of the theorem is trivially true if $1 \le n \le r$. Moreover, if $n = r+1$, then $e(\h) \le r$ because otherwise, $\h = K_{r+1}^r$ and then it is easy to see that there is a (Hamiltonian) Berge cycle of length $r+1$ in $\h$, a contradiction. Therefore, $e(\h) \le r= n-1$. Moreover, equality holds if and only if $\partial_2(\h) = K_{r+1}$ and $\h$ consists of $r$ hyperedges. 

We will show the statement is true for $n$ assuming it is true for all smaller values. Let $\h$ be an $r$-uniform hypergraph on $n$ vertices having no Berge cycle of length $r+1$ or longer. We show that we may assume the following two properties hold for $\h$.

\begin{enumerate}
\item[(1)] For any set $S \subseteq V(\h)$ with $|S| \le |V(\h)|-1 = n-1$, the number of hyperedges of $\h$ incident to the vertices of $S$ is at least $|S|$. 

Indeed, suppose there is a set $S \subset V(\h)$ (i.e., $|S| \le |V(\h)|-1$) with fewer than $|S|$ hyperedges incident to the vertices of $S$. We delete the vertices of $S$ from $\h$ to obtain a new hypergraph $\h'$ on $n-|S|$ vertices. By induction, $\h'$ contains at most $(n-|S|-1)$ hyperedges, so $\h$ contains less than $(n-1-|S|)+|S| = (n-1)$ hyperedges, as required.


\item[(2)] There is no cut-hyperedge in $\h$.

Indeed, if $h \in E(\h)$ is a cut-hyperedge, then $\partial_2(\h \setminus \{h\})$ is not a connected graph, so there are disjoint non-empty sets $V_1$ and $V_2$ such that $V(\h) = V_1 \cup V_2$ and there are no edges of $\partial_2(\h \setminus \{h\})$ between $V_1$ and $V_2$. So the hypergraphs $\h[V_1]$ and $\h[V_2]$ do not contain a Berge cycle of length $r+1$ or longer. Therefore, by induction, $e(\h[V_1]) \le |V_1|-1$ and $e(\h[V_2]) \le |V_2|-1$. In total, $e(\h) = e(\h[V_1]) + e(\h[V_2]) + 1 \le (|V_1|+|V_2|-2)+1 = |V(\h)|-1$, as desired.


Moreover, we claim that the equality $e(\h) = |V(\h)|-1$ cannot hold in this case (i.e., if there is a cut-hyperedge). Indeed, if equality holds, then we must have $e(\h[V_1]) = |V_1|-1$ and $e(\h[V_2]) = |V_2|-1$. 
Notice that since $r \ge 3$, the hyperedge $h$ either contains at least two vertices $x, y \in V_1$ or two vertices $x, y \in V_2$. Without loss of generality, assume the former is true. By induction, $\partial_2(\h[V_1])$ is connected and for every block $D$ of $\partial_2(\h[V_1])$, we have $D = K_{r+1}$ and the subhypergraph induced by $D$ consists of $r$ hyperedges. So by Lemma \ref{r_path_between_any_pair}, there is a Berge path of length $r$ (consisting of the $r$ hyperedges induced by $D$) between any two vertices of a block $D$. Then it is easy to see that since $\partial_2(\h[V_1])$ is connected, there is a Berge path $\P$ of length at least $r$ between any two vertices of $V_1$, so in particular between $x$ and $y$. Then $\P$ together with $h$ forms a Berge cycle of length $r+1$ in $\h$, a contradiction.
\end{enumerate}

Consider an auxiliary bipartite graph $B$ consisting of vertices of $\h$ in one class and hyperedges of $\h$ on the other class. Then property (1) shows that Hall's condition holds for all subsets of $V(\h)$ of size up to $|V(\h)|-1$. Therefore, there is a matching in $B$ that matches all the vertices in $V(\h)$, except at most one vertex, say $x$. In other words, there exists an injection $f : V(\h)\setminus \{x\} \to E(\h)$ such that for every $v \in V(\h) \setminus \{x\}$, we have $v \in f(v)$. Given an injection $f : V(\h) \setminus \{x\} \to E(\h)$ with $v \in f(v)$, let $\P_f$ be a longest Berge path of the form $v_1f(v_1)v_2f(v_2) \ldots v_{l-1}f(v_{l-1})v_l$  where for each $1 \le i \le l-1$, $v_{i+1} \in f(v_i)$. Moreover, among all injections $f : V(\h) \setminus \{x\} \to E(\h)$ with $v \in f(v)$, suppose $\phi : V(\h) \setminus \{x\} \to E(\h)$ is an injection for which the path $\P_{\phi} = v_1\phi(v_1)v_2\phi(v_2) \ldots v_{l-1}\phi(v_{l-1})v_l$ is a longest path. 

Because of the way $\P_{\phi}$ was constructed, it is also clear that $x \not \in \{v_1, v_2, \ldots, v_{l-1} \}$. We consider two cases depending on whether $v_l$ is equal to $x$ or not. 

\vspace{2mm}

\textbf{Case 1: $v_l \not = x$.} Our aim is to get a contradiction, and show that this case is impossible.

\begin{claim}
\label{last_edge}
If $v_l \not = x$, then $\phi(v_l) = \{v_{l-r+1}, v_{l-r+2}, \ldots, v_l\}$.
\end{claim}
\begin{proof}
If $v_l \not = x$, then we claim $\phi(v_l) = \{v_{l-r+1}, v_{l-r+2}, \ldots, v_l\}$. Indeed, if $\phi(v_l)$ contains a vertex $v_i \in \{v_1, v_2, \ldots, v_{l-r}\}$, then the Berge cycle $v_i\phi(v_i)v_{i+1}\phi(v_{i+1}) \ldots v_l \phi(v_l) v_i$ is of length $r+1$ or longer, a contradiction. Moreover, if $\phi(v_l)$ contains a vertex $v \not \in \{v_1, v_2, \ldots, v_l\}$, then $P_{\phi}$ can be extended to a longer path $v_1\phi(v_1)v_2\phi(v_2), \ldots, v_{l-1}\phi(v_{l-1})v_l\phi(v_l)v$, a contradiction again, proving that $\phi(v_l) = \{v_{l-r+1}, v_{l-r+2}, \ldots, v_l\}$.
\end{proof}

Fix some $i \in \{l-r+1, l-r+2, \ldots, l-1\}$. Let us define a new injection $\psi: V(\h) \setminus \{x\} \to E(\h)$ as follows: $\psi(v) = \phi(v)$ for every $v \not \in \{x, v_1,v_2, \ldots, v_l\}$, and for every $v \in \{v_1, v_2, \ldots, v_{i-1}\}$. Moreover, let $\psi(v_i) = \phi(v_l)$ and $\psi(v_k) = \phi(v_{k-1})$ for each $l \ge k \ge i+1$.  Now consider the Berge path $v_1\phi(v_1)v_2\phi(v_2)$ $\ldots$ $v_i\phi(v_l)v_l\phi(v_{l-1}) \ldots v_{i+2}\phi(v_{i+1})v_{i+1}$ $= v_1\psi(v_1)v_2\psi(v_2)$ $\ldots$ $v_i\psi(v_i)v_l\psi(v_{l}) \ldots v_{i+2}\psi(v_{i+2})v_{i+1}$. This path has the same length as $\P_{\phi}$, so it is also a longest path. Moreover, $v_{i+1} \not = x$, so we can apply Claim \ref{last_edge} to conclude that $\psi(v_{i+1}) = \{v_{l-r+1}, v_{l-r+2}, \ldots, v_l\} = \phi(v_i)$. But then $\phi(v_i) = \phi(v_l)$, a contradiction to the fact that $\phi$ was an injection. 

\vspace{2mm}

\textbf{Case 2: $v_l = x$.} 

\begin{claim}
\label{r+1setagain}
$\phi(v_{l-1}) \subset \{v_{l-r}, v_{l-r+1}, \ldots, v_l\}.$ 
\end{claim}

\begin{proof}
If $\phi(v_{l-1})$ contains a vertex $v \not \in \{v_1, v_2, \ldots, v_l \}$, then we consider the Berge path $v_1\phi(v_1)v_2\phi(v_2)$$,\ldots,$$ v_{l-1}\phi(v_{l-1})v$. Since $v \not = x$, we get a contradiction by Case 1. Moreover, if $\phi(v_{l-1})$ contains a vertex $v_i$ with $i \in \{1,2, \ldots, l-r-1\}$, then the Berge cycle $v_i \phi(v_i) v_{i+1} \phi(v_{i+1}) \ldots v_{l-1} \phi(v_{l-1}) v_i$ is of length $r+1$ or longer, a contradiction. This finishes the proof of the claim. 
\end{proof}

By Claim \ref{r+1setagain}, we know that $\phi(v_{l-1}) = \{v_{l-r}, v_{l-r+1}, \ldots, v_{l-1}, v_l\} \setminus \{v_j\}$ for some $j$ with $l-r \le j \le l-2$. (From now, in the rest of the proof we fix this $j$.)

\begin{claim}
\label{finding_many_sets}
For any $i \in \{l-r, l-r+1, \ldots, l-1\} \setminus \{j\}$, we have $\phi(v_i) \subset \{v_{l-r}, v_{l-r+1}, \ldots, v_{l-1}, v_l\}$.
\end{claim} 
\begin{proof}
When $i =l-1$, we know the statement is true by Claim \ref{r+1setagain}. 

Suppose $i \in \{l-r, l-r+1, \ldots, l-2\} \setminus \{j\}$. Let us define a new injection $\psi: V(\h) \setminus \{x\} \to E(\h)$ as follows: $\psi(v) = \phi(v)$ for every $v \not \in \{v_1,v_2, \ldots, v_l\}$, and for every $v \in \{v_1, v_2, \ldots, v_{i-1}\}$. Moreover, let $\psi(v_i) = \phi(v_{l-1})$ and $\psi(v_k) = \phi(v_{k-1})$ for each $l-1 \ge k \ge i+1$.  Now consider the Berge path $v_1\phi(v_1)v_2\phi(v_2)$ $\ldots$ $v_i\phi(v_{l-1})v_{l-1}\phi(v_{l-2})$$ \ldots $$v_{i+1}$ $= v_1\psi(v_1)v_2\psi(v_2)$ $\ldots$ $v_i\psi(v_i)v_{l-1}\psi(v_{l-1}) \ldots v_{i+1}$. (Note that when $i = l-2$, the Berge path is simply $v_1\phi(v_1)v_2\phi(v_2)$ $\ldots$ $v_i\phi(v_{l-1})v_{l-1} = v_1\psi(v_1)v_2\psi(v_2) \ldots v_i\psi(v_i)v_{l-1}$.)

If $\psi(v_{i+1})$ contains a vertex $v \not \in \{v_1, v_2, \ldots, v_l \}$, then the Berge path $v_1\psi(v_1)v_2\psi(v_2)$ $\ldots$ $v_i\psi(v_i)v_{l-1}\psi(v_{l-1}) \ldots v_{i+2}\psi(v_{i+2})v_{i+1}\psi(v_{i+1})v$ has the same length as $\P_{\phi}$, so it is also a longest path. Moreover, since $v \not = x$, we get a contradiction by Case 1. 

If $\psi(v_{i+1})$ contains a vertex $v_k \in \{v_1, v_2, \ldots, v_{l-r-1}\}$ then one can see that the Berge cycle $v_k \psi(v_k) v_{k+1} \psi(v_{k+1})$ $\ldots$ $v_{l-1} \psi(v_{l-1}) v_k$ is of length $r+1$ or longer, a contradiction. Therefore, we have $\psi(v_{i+1}) \subset \{v_{l-r}, v_{l-r+1}, \ldots, v_l\}$. But we defined $\psi(v_{i+1}) = \phi(v_i)$, proving the claim. 
\end{proof}

Note that Claim \ref{finding_many_sets} shows that $r-1$ hyperedges of $\h$ are contained in a set $S := \{v_{l-r}, v_{l-r+1},\ldots, v_{l-1}, v_l\}$ of size $r+1$. The following claim shows that if we can find one more hyperedge of $\h$ contained in $S$, then $S$ must induce a block of $\partial_2(\h)$. 

\begin{claim}
\label{findingoneblock}
Suppose $r \ge 3$. If a set $S$ of size $r+1$ contains $r$ hyperedges of $\h$ then it induces a induces a block of $\partial_2(\h)$.
\end{claim}

\begin{proof}
Since the set $S$ contains at least $3$ hyperedges every pair $x, y \in S$ is contained in some hyperedge. Thus $\partial_2(\h[S]) = K_{r+1}$. Consider a (maximal) block $D$ of $\partial_2(\h)$ containing $S$. 

Suppose $D$ contains a vertex $t \not \in S$. Then since $D$ is 2-connected, there are two paths $P_1, P_2$ in $\partial_2(\h)$ between $t$ and $S$, which are vertex-disjoint besides $t$. Let $V(P_1) \cap S = \{u\}$ and $V(P_2) \cap S = \{v\}$. For each edge $xy \in E(P_1) \cup E(P_2)$, fix an arbitrary hyperedge $h_{xy}$ of $\h$ containing $xy$. It is easy to see that a subset of the hyperedges $\{h_{xy} \mid xy \in E(P_1) \cup E(P_2)\}$ forms a Berge path $\P$ between $u$ and $v$.

On the other hand, by Lemma \ref{r_path_between_any_pair}, there is a Berge path $\P'$ of length $r$ between $u$ and $v$ consisting of the $r$ hyperedges contained in $S$. Note that $\P$ and $\P'$ do not share any hyperedges (indeed, each hyperedge of $\P$ contains a vertex not in $S$, while hyperedges of $\P'$ are contained in $S$). Therefore, $\P$ together with $\P'$ forms a Berge cycle of length $r+1$ or longer, a contradiction. Therefore, $D$ contains no vertex outside $S$; thus $S$ induces a block of $\partial_2(\h)$, as required.
\end{proof}

We will use the above claim several times later. At this point we need to distinguish the cases $r =3$ and $r \ge 4$, since Lemma \ref{r-1pathbetweenanypair} only applies in the latter case. 

\subsection*{The case $r \ge 4$}
Since $r \ge 4$, by Claim \ref{finding_many_sets} and Lemma \ref{r-1pathbetweenanypair} there is a Berge path of length $r-1$ between any two vertices of $S = \{v_{l-r}, v_{l-r+1},\ldots, v_{l-1}, v_l\}$. This will allow us to show the following.
\begin{claim}
\label{onemoresetinS}
$\phi(v_j) \subset \{v_{l-r}, v_{l-r+1}, \ldots, v_{l-1}, v_l \} = S$
\end{claim}
\begin{proof}
Suppose for a contradiction that $\phi(v_j)$ contains a vertex $v \not \in S$. The hyperedge $\phi(v_j)$ contains at least two vertices from $S$, namely $v_j$ and $v_{j+1}$. By property (2), $\phi(v_j)$ is not a cut-hyperedge of $\h$. So after deleting $\phi(v_j)$ from $\h$, the hypergraph $\h \setminus \{\phi(v_j)\}$ is still connected -- so there is a (shortest) Berge path $\Q$ in $\h \setminus \{\phi(v_j)\}$ between $v$ and a vertex $s \in S$ (note that the hyperedges of $\Q$ are not contained in $S$). The vertex $s$ is different from either $v_j$ or $v_{j+1}$, say $s \not = v_j$, without loss of generality. By Lemma \ref{r-1pathbetweenanypair}, there is a Berge path $\Q'$ of length $r-1$ between $s$ and $v_j$ (consisting of the hyperedges contained in $S$). Then $\Q, \Q'$ and $\phi(v_j)$ form a Berge cycle of length at least $r+1$ in $\h$, a contradiction. 
\end{proof}

Claim \ref{finding_many_sets} and Claim \ref{onemoresetinS} together show that there are at least $r$ hyperedges of $\h$ contained in $S$. If all $r+1$ subsets of $S$ of size $r$ are hyperedges of $\h$, then $S$ induces $K_{r+1}^r$ and it is easy to show that it contains a Berge cycle of length $r+1$, a contradiction. This means $S$ contains exactly $r$ hyperedges of $\h$. Then by Claim \ref{findingoneblock}, we know that $S$ induces a block of $\partial_2(\h)$. 



Let $D_1, D_2, \ldots, D_p$ be the unique decomposition of $\partial_2(\h)$ into $2$-connected blocks. Claim \ref{findingoneblock} shows that one of these blocks, say $D_1$, is induced by $S$. Let us contract the vertices of $S$ to a single vertex, to produce a new hypergraph $\h'$. Then it is clear that the block decomposition of $\partial_2(\h')$ consists of the blocks $D_2, \ldots, D_p$. So $\h'$ does not contain any Berge cycle of length $r+1$ or longer, as well; moreover, $|V(\h')| = |V(\h)|-r$ and $e(\h') = e(\h)-r$. By induction, we have  $e(\h') \le |V(\h')|-1$. Therefore, $$e(\h) = e(\h') + r \le (|V(\h')|-1) +r = (|V(\h)|-r-1)+r = |V(\h)|-1.$$

If $e(\h) =  |V(\h)|-1$, then we must have $e(\h') = |V(\h')|-1$ and $S$ must contain exactly $r$ hyperedges. Moreover, since equality holds for $\h'$, by induction, $\partial_2(\h')$ is connected and for each block $D_i$ (with $2 \le i \le p$) of $\partial_2(\h')$, $D_i = K_{r+1}$ and $\h'[D_i]$ contains exactly $r$ hyperedges. This means that for every block $D$ of $\partial_2(\h)$, we have $D = K_{r+1}$ and $\h[D]$ contains exactly $r$ hyperedges, completing the proof in the case $r\ge 4$.

\subsection*{The case $r = 3$}

	Recall that using Claim \ref{finding_many_sets} we can find a set $S$ of size $4$ which contains $2$ hyperedges of $\h$. 
Let $S = \{x, y, a, b\}$ and the two hyperedges be $xab$ and $yab$. By property (2), $xab$ is not a cut-hyperedge of $\h$.  So after deleting $xab$ from $\h$, the hypergraph $\h \setminus \{xab\}$ is still connected -- so there is a (shortest) Berge path $\mathcal Q$ between $x$ and $\{y,a,b\}$. If $\mathcal Q$ is of length at least $2$, then it is easy to see that $\mathcal Q$ together with $yab$ and $xab$ form a Berge cycle of length at least $4$, a contradiction. So $\mathcal Q$ consists of only one hyperedge, say $h$. 

Our goal is to find a set of vertices which induces a block of $\partial_2(\h)$, so that we can apply induction.

\vspace{2mm}

If $\abs{h \cap \{y,a,b\}} = 2$ then $h, xab, yab$ are $3$ hyperedges of $\h$ contained in $S$, so by Claim \ref{findingoneblock}, we can conclude that $S$ induces a block of $\partial_2(\h)$. (Notice that $S$ contains exactly $\abs{S}-1 = 3$ hyperedges of $\h$, otherwise it is easy to find a Berge cycle of length 4; this will be useful later.) So we can suppose $\abs{h \cap \{y,a,b\}} = 1$. We consider two cases depending on whether $h$ is either $xat$ or $xbt$, or whether $h$ is $xyt$ for some $t \not \in S$.

\vspace{2mm}
    
    \textbf{Case 1.} First suppose without loss of generality that $h = xat$ for some $t \not \in S$. Consider the set $\mathcal D$ of all hyperedges of $\h$ containing the pairs $xa$, $ab$ or $xb$ and let $D$ be the set of vertices spanned by them. For each pair of vertices $i, j \in \{x,a,b\}$, let $V_{ij} = \{v \mid ijv \in \h \} \setminus \{x,a,b\}$. We claim that the sets $V_{xa}, V_{ab}, V_{xb}$ are pairwise disjoint. Suppose for the sake of a contradiction that $t' \in V_{xa} \cap V_{ab}$. Then the hyperedges $xat',  abt', xab$ are contained in a set of $4$ vertices $\{x,a,b,t'\}$. Thus by Claim \ref{findingoneblock}, this set induces a block of $\partial_2(\h)$ and we are done (we found the desired block!). Thus we can suppose $ V_{xa} \cap V_{ab} = \emptyset$. Similarly $V_{ab} \cap V_{xb} = \emptyset$ and $V_{xa} \cap V_{xb} = \emptyset$. This shows that $\abs{D} = 3 + \abs{V_{xa}} + \abs{V_{xb}} + \abs{V_{ab}}$. On the other hand, $\mathcal D$ consists of $1+\abs{V_{xa}} + \abs{V_{xb}} + \abs{V_{ab}}$ hyperedges, so $\abs{\mathcal D} = \abs{D}-2$.  
	
	We will now show that $D$ induces a block of $\partial_2(\h)$. 
Let $D'$ be a (maximal) block of $\partial_2(\h)$ containing $D$ and suppose for the sake of a contradiction that it contains a vertex $p \not \in D$. Then since $D'$ is $2$-connected, there are two paths $P_1, P_2$ in $\partial_2(\h)$ between $p$ and $D$, which are vertex-disjoint besides $p$. Let $V(P_1) \cap D = \{u\}$ and $V(P_2) \cap D = \{v\}$. For each edge $xy \in E(P_1) \cup E(P_2)$, fix an arbitrary hyperedge $h_{xy}$ of $\h$ containing $xy$. It is easy to see that a subset of the hyperedges $\{h_{xy} \mid xy \in E(P_1) \cup E(P_2)\}$ forms a Berge path $\P$ between $u$ and $v$. If $uv \not \in \{xa, ab, xb \}$, then it is easy to see that there is a path $\mathcal P'$ of length 3 between $u$ and $v$ consisting of the hyperedges of $\mathcal D$. Then $\P$ together with $\mathcal P'$ forms a Berge cycle of length at least 4 in $\h$, a contradiction. On the other hand if $uv \in \{xa, ab, xb \}$, then $\P$ must contain at least two hyperedges of $\h$ because otherwise $\P = \{puv\}$ but then $puv$ should have been in $\mathcal D$ (since by definition $\mathcal D$ must contain all the hyperedges of $\h$ containing the pair $uv$); moreover, it is easy to check that between $u$ and $v$ there is a Berge path $\mathcal P'$ of length 2 consisting of the hyperedges of $\mathcal D$. Then again, $\P$ together with $\mathcal P'$ forms a Berge cycle of length at least 4 in $\h$, a contradiction. Therefore, $D'$ contains no vertex outside $D$; so $D$ induces a block of $\partial_2(\h)$ (which contains $\abs{D}-2$ hyperedges of $\h$), as desired.

\vspace{2mm}
 
\textbf{Case 2.} Finally suppose $h = xyt$ for some $t \not \in S$. Let $\mathcal D$ be the set of all hyperedges of $\h$ containing the pair $xy$ plus the hyperedges $xab$ and $yab$, and let $D$ be the set of vertices spanned by the hyperedges of $\mathcal D$. Let $V_{xy} = \{v \mid xyv \in \h \}$. We claim that $a \not \in V_{xy}$ and $b \not \in V_{xy}$. Indeed suppose for the sake of a contradiction that $a \in V_{xy}$. Then the hyperedges $xab, yab, xya$ are contained in a set of 4 vertices $\{x,y,a,b\}$. So by Claim \ref{findingoneblock}, this set induces a block of $\partial_2(\h)$, and we are done. So $a \not \in V_{xy}$. Similarly, we can conclude $b \not \in V_{xy}$. Therefore, $\abs{D} = \abs{V_{xy}}+4$. On the other hand, $\abs{\mathcal D} = \abs{V_{xy}}+2$, so $\abs{\mathcal D} = \abs{D}-2$. 

We claim that $D$ induces a block of $\partial_2(\h)$. The proof is very similar to that of \textbf{Case 1}, we still give it for completeness. Let $D'$ be a (maximal) block of $\partial_2(\h)$ containing $D$ and suppose for the sake of a contradiction that it contains a vertex $p \not \in D$. Then since $D'$ is $2$-connected, there are two paths $P_1, P_2$ in $\partial_2(\h)$ between $p$ and $D$, which are vertex-disjoint besides $p$. Let $V(P_1) \cap D = \{u\}$ and $V(P_2) \cap D = \{v\}$. For each edge $xy \in E(P_1) \cup E(P_2)$, fix an arbitrary hyperedge $h_{xy}$ of $\h$ containing $xy$. It is easy to see that a subset of the hyperedges $\{h_{xy} \mid xy \in E(P_1) \cup E(P_2)\}$ forms a Berge path $\P$ between $u$ and $v$. 

If $uv \not = xy$, then it is easy to see that there is a path $\mathcal P'$ of length 3 or 4 between $u$ and $v$ consisting of the hyperedges of $\mathcal D$. (Indeed if $u,v \in V_{xy}$, then $\mathcal P'$ is of length 4, otherwise it is of length 3.) Then $\P$ together with $\mathcal P'$ forms a Berge cycle of length at least 4 in $\h$, a contradiction. On the other hand if $uv = xy$, then $\P$ must contain at least two hyperedges of $\h$ because otherwise $\P = \{puv\}$ but then $puv$ should have been in $\mathcal D$ (since by definition $\mathcal D$ must contain all the hyperedges of $\h$ containing the pair $uv$); moreover, it is easy to check that between $u$ and $v$ there is a Berge path $\mathcal P'$ of length 2 consisting of the hyperedges of $\mathcal D$. Then again, $\P$ together with $\mathcal P'$ forms a Berge cycle of length at least 4 in $\h$, a contradiction. Therefore, $D'$ contains no vertex outside $D$; so $D$ induces a block of $\partial_2(\h)$ (and contains $\abs{D}-2$ hyperedges of $\h$), as desired. 


\vspace{2mm}

Let $D_1, D_2, \ldots, D_p$ be the unique decomposition of $\partial_2(\h)$ into $2$-connected blocks. In \textbf{Case 1} and \textbf{Case 2} we showed that one of these blocks, (say) $D_1 = D$ is such that $\h[D_1]$ contains $\abs{D_1}-2$ hyperedges of $\h$, otherwise, $D_1$ is a set of $4$ vertices such that $\h[D_1]$ contains exactly $\abs{D_1}-1 = 3$ hyperedges of $\h$. In all these cases, note that $e(\h[D_1]) \le \abs{D_1}-1$.

Let us contract the vertices of $D_1$ to a single vertex, to produce a new hypergraph $\h'$. Then it is clear that the block decomposition of $\partial_2(\h')$ consists of the blocks $D_2, \ldots, D_p$. So $\h'$ does not contain any Berge cycle of length $4$ or longer, as well; moreover, $|V(\h')| = |V(\h)|-|D_1|+1$ and $e(\h') = e(\h)-e(\h[D_1])$. By induction, we have  $e(\h') \le |V(\h')|-1$. Therefore, $$e(\h) = e(\h') + e(\h[D_1]) \le |V(\h')|-1 +\abs{D_1}-1 = (|V(\h)|-|D_1|+1)-1+\abs{D_1}-1 = |V(\h)|-1.$$

If $e(\h) =  |V(\h)|-1$, then we must have $e(\h') = |V(\h')|-1$ and $\h[D_1]$ must contain exactly $\abs{D_1}-1$ hyperedges. As noted before, this is only possible if $D_1$ has $4$ vertices and induces exactly $3$ hyperedges of $\h$. Moreover, since equality holds for $\h'$, by induction, $\partial_2(\h')$ is connected and for each block $D_i$ (with $2 \le i \le p$) of $\partial_2(\h')$, $D_i = K_{4}$ and $\h'[D_i]$ contains exactly $3$ hyperedges. This means for every block $D$ of $\partial_2(\h)$, we have $D = K_{4}$ and $\h[D]$ contains exactly $3$ hyperedges of $\h$, completing the proof in the case $r = 3$.

\section*{Acknowledgment}

The research of the authors is partially supported by the National Research, Development and Innovation Office  NKFIH, grant K116769.

\end{document}